\documentclass[11pt]{article} 
\usepackage{caption}
\usepackage[utf8]{inputenc} 
\usepackage{amsmath}
\usepackage{amssymb}
\usepackage[dvipsnames]{xcolor}
\definecolor{mygray}{gray}{0.6}
\usepackage{mathtools}
\usepackage{leftindex}
\usepackage[version=4]{mhchem}
\usepackage{floatpag}
\usepackage[dvipsnames]{xcolor}

\usepackage[toc,page]{appendix}
\makeatletter
\newcommand*\bigcdot{\mathpalette\bigcdot@{.5}}
\newcommand*\bigcdot@[2]{\mathbin{\vcenter{\hbox{\scalebox{#2}{$\m@th#1\bullet$}}}}}
\makeatother

\usepackage{ulem}

\usepackage{amsthm}
\usepackage{ mathrsfs }
\usepackage{placeins}
\usepackage{titlesec}
\titleformat{\paragraph}
{\normalsize\bfseries}{\theparagraph}{1em}{}
\titlespacing*{\paragraph}
{0pt}{3.25ex plus 1ex minus .2ex}{1.5ex plus .2ex}

\titleclass{\subsubsubsection}{straight}[\subsection]

\newcounter{subsubsubsection}[subsubsection]
\renewcommand\thesubsubsubsection{\thesubsubsection.\arabic{subsubsubsection}}
\renewcommand\theparagraph{\thesubsubsubsection.\arabic{paragraph}} 

\titleformat{\subsubsubsection}
  {\sffamily\small}{\thesubsubsubsection}{1em}{}
\titlespacing*{\subsubsubsection}
{0pt}{3.25ex plus 1ex minus .2ex}{1.5ex plus .2ex}

\newcommand{\RN}[1]{%
  \textup{\uppercase\expandafter{\romannumeral#1}}%
}
\usepackage{titlesec}
\titleformat{\section}{\normalfont\Large\bfseries}{\S\thesection}{1em}{}[]

\newtheorem{Theorem}{Theorem}[section]

\newtheorem{Lemma}[Theorem]{Lemma}

\newtheorem{Remark}[Theorem]{Remark}

\usepackage{float}
\restylefloat{table}
\usepackage{dblfloatfix}
\usepackage[hidelinks]{hyperref}
\usepackage{color}
\usepackage[table,xcdraw]{xcolor}

\usepackage{placeins}
  \usepackage{graphicx}
    \usepackage{caption}    
    \usepackage{array,booktabs} \usepackage{amsthm,amsmath,amsfonts,amssymb,latexsym,amscd,mathrsfs}
    \usepackage{textcomp,float}
    \usepackage[margin=25mm]{geometry}

    \usepackage{tikz}
    \usepackage{tikz-3dplot}
    \usetikzlibrary{shapes,calc,positioning}
\usepackage{afterpage}
\usepackage{placeins}
\usepackage{hyperref}
\def\bF{\mathbb{F}}
\def\Fq{\mathbb{F}_q}
\def\GL{\mathrm{GL}}
\def\PG{\mathrm{PG}}
\def\PGL{\mathrm{PGL}}
\usepackage{graphicx} 
\definecolor{amber(sae/ece)}{rgb}{1.0, 0.49, 0.0}
	\definecolor{darkcyan}{rgb}{0.0, 0.55, 0.55}
	\definecolor{darkseagreen}{rgb}{0.56, 0.74, 0.56}
	\definecolor{salmon}{rgb}{1.0, 0.55, 0.41}
\usepackage{booktabs} 
\usepackage{array} 
\usepackage{paralist} 
\usepackage{verbatim} 
\usepackage{subfig} 
\definecolor{ForestGreen}{RGB}{34,139,34}
\usepackage{fancyhdr} 
\definecolor{carminepink}{rgb}{0.92, 0.3, 0.26}
	\definecolor{electricyellow}{rgb}{1.0, 1.0, 0.0}
	\definecolor{chromeyellow}{rgb}{1.0, 0.65, 0.0}
	\definecolor{babyblue}{rgb}{0.54, 0.81, 0.94}

\usepackage{sectsty}
\allsectionsfont{\sffamily\mdseries\upshape} 

 \def\bF{\mathbb{F}}
\def\Fq{\mathbb{F}_q}
\def\GL{\mathrm{GL}}

\def\F2h{\mathbb{F}_{2^h}}

\def\PG{\mathrm{PG}}
\def\PGL{\mathrm{PGL}}

\def\cP{\mathcal{P}}

\def\cV{\mathcal{V}}
\def\cZ{\mathcal{Z}}
\def\cC{\mathcal{C}}
\def\cH{\mathcal{H}}

\def\cN{\mathcal{N}}

\usepackage[nottoc,notlof,notlot]{tocbibind} 
\usepackage[titles,subfigure]{tocloft}

\title{Nets of conics containing a double line in $\PG(2,q)$, $q$ even}

\author{Nour Alnajjarine\footnote{University  of Rijeka (\texttt{nour.alnajjarine@math.uniri.hr})} \;  and \;
Michel Lavrauw\footnote{University of Primorska; Vrije Universiteit Brussel (\texttt{michel.lavrauw@famnit.upr.si})}}

\begin{document} 
\maketitle
\begin{abstract}

 This paper completes the classification of nets of conics containing at least one double line in $\PG(2,q)$ for $q$ even. This classification contributes to the classification of partially symmetric tensors in $\mathbb{F}_q^3 \otimes S^2 \mathbb{F}_q^3$, $q$ even.
The proof is obtained using geometric and combinatorial properties of the Veronese surface in 5-dimensional projective space over the finite field of even order.
In particular, the orbits of planes in $\PG(5,q)$ that intersect the nucleus plane of the Veronese surface in at least one point are classified. 
As a result, it is shown that there are exactly $18$ equivalence classes of nets in $\PG(2,q)$, $q$ even, containing at least one double line, $9$ of which have an empty base. 
\end{abstract}

\textbf{Keywords:} 
Nets of Conics,\hspace{0.2cm}  Double Lines, \hspace{0.2cm} Nucleus Plane,\hspace{0.2cm} Veronese Surface\hspace{0.2cm}

\section{Introduction}

A linear system of conics is defined as a subspace of the projective space of quadratic forms in $\bF[X,Y,Z]$. The classification of linear systems of conics over finite fields was first addressed by Dickson in \cite{dickson}, who classified pencils of conics in $\mathrm{PG}(2, q)$ for $q$ odd using a purely algebraic approach. Following a similar approach, Campbell and Wilson partially classified pencils of conics over finite fields of even characteristic, and nets of conics in \cite{campbell,campbell2,wilson}. 

In 2020, a new framework for classifying these linear systems was introduced in \cite{lines,nets}, combining algebraic methods with geometric and combinatorial insights, which has since led to significant progress.
For the classification and characterization of pencils, webs, and squabs of conics, and non-empty base nets of conics in $\mathrm{PG}(2, q)$, as well as an explanation of some of the shortcomings in Campbell's and Wilson's partial classifications, we refer the reader to \cite{lines, solidsqeven, nets, planesqeven, webs}. We also refer to \cite{AbEmIa}, which examines nets over algebraically closed fields of characteristic zero in connection with Artinian algebras of length 7, and provides, in Appendices A and B, an interesting historical overview of study of nets of conics. For the classification of nets of conics over $\mathbb{R}$, we refer to the work of Wall \cite{wall}.

The approach in \cite{lines,nets} lifts the problem of classifying linear systems of conics to a higher-dimensional framework, where the properties of the Veronese surface $\mathcal{V}(\mathbb{F}_q)$ in $\mathrm{PG}(5, q)$ are used to distinguish projectively inequivalent linear systems using a combination of geometric and combinatorial invariants, such as e.g. the {\it orbit distributions} (see Section \ref{pre}). For further insight into the geometric and algebraic properties of Veronese varieties over fields of non-zero characteristic, we refer the reader to \cite{Havlicek}.

For $q$ odd, there exists a polarity of $\PG(5,q)$ that maps the set of conic planes of $\mathcal{V}(\Fq)$ onto the set of tangent planes of $\mathcal{V}(\Fq)$ (see, e.g., \cite[Theorem 4.25]{galois geometry}). This induces a correspondence between nets of conics in $\PG(2,q)$ containing at least one double line and planes in $\PG(5,q)$ that meet $\cV(\Fq)$ in at least one point, for $q$ odd. However, this correspondence does not hold over finite fields of characteristic $2$, as shown in \cite[Section 4]{planesqeven}. In light of this correspondence, nets in $\PG(2,q)$, $q$ odd, containing at least one double line were referred to as \textit{rank-one nets}. Rank-one nets of conics in $\PG(2,q)$, with $q$ odd, were classified in \cite{nets}, resulting in $15$ orbits under $\PGL(3,q)$. 

In this paper, we complete the classification of nets of conics containing at least one double line in $\PG(2,q)$, $q$ even, by classifying $K$-obits of planes in $\PG(5,q)$ intersecting the \textit{nucleus plane} non-trivially, where $K$ is the subgroup in $\PGL(6,q)$ stabilizing the Veronese surface $\mathcal{V}(\mathbb{F}_q)$.  Moreover, for each of the orbits we determine the so-called point-orbit distributions, which are combinatorial invariants for these orbits.

\begin{Theorem}\label{mainresult}

There are $18$ $\PGL(3,q)$-orbits of nets of conics in $\PG(2,q)$ containing at least one double line for $q$ even.
Representatives and point-orbit distributions of the corresponding orbits of planes in $\PG(5,q)$ are as listed in Table \ref{tableplanesintersectingpiN}.
\end{Theorem}

Our result contributes to the classification of partially symmetric tensors in $\mathbb{F}_q^3 \otimes S^2 \mathbb{F}_q^3$ for even $q$. In this context, classifying tensors in $\mathbb{F}_q^3 \otimes S^2 \mathbb{F}_q^3$ is equivalent to determining the $K$-orbits of points, lines, and planes in $\mathrm{PG}(5,q)$. Although a complete classification of tensors in $\Fq^3 \otimes \Fq^3 \otimes \Fq^3$ is currently beyond reach, progress has been achieved by starting with the study of partially symmetric representations of tensors in $\Fq^3 \otimes \Fq^3 \otimes \Fq^3$. For further details on this classification problem and on the geometry associated with tensors in $\Fq^r \otimes \Fq^3 \otimes \Fq^3$, $r \leq 6$, we refer the reader to \cite{canonical,lines,planesqeven,solidsqeven}. We also refer the reader to the package {\it T233}, implemented in GAP \cite{NourMichel,T233,GAP,fining}, which employs the geometry of the associated contraction spaces to determine the ranks and orbits of tensors in $\Fq^2 \otimes \Fq^3 \otimes \Fq^3$.

The structure of the paper is as follows. In Section \ref{pre}, we present the necessary definitions and theoretical background for our main result. The proof of Theorem \ref{mainresult} is divided into two parts: Section \ref{main1} addresses nets with at least one double line and a non-empty base, while Section \ref{main2} considers nets with at least one double line and an empty base, where the {\it base} of a net of conics is defined as the intersection points of the conics contained in the net. 
In both cases, the approach is using the equivalent setting in the associated Veronese space $\PG(5,q)$.
Finally, in Section \ref{comparison}, we compare our classification with the partial classification of nets as presented in \cite{campbell2}.

\section{Preliminaries}\label{pre}

In this section, we present the necessary definitions and theoretical background for our proofs. Some of the results are classical and can be found in \cite{berlekamp,galois geometry,hirsch}, while others are more recent contributions from \cite{webs,solidsqeven,planesqeven,roots,lines,nets,LaPoSh2021}. Throughout this work, the finite field of order $q$ is denoted by $\mathbb{F}_q$, the $n$-dimensional projective space over $\mathbb{F}_q$ is denoted by $\mathrm{PG}(n, q)$ and the indeterminates of the coordinate rings of $\PG(2,q)$ and $\PG(5,q)$ are represented by $(X,Y,Z)$ and $(Y_0,\ldots,Y_5)$, respectively. The trace map from $\mathbb{F}_q$ to its prime subfield $\mathbb{F}_p$ is denoted by $\operatorname{Tr}$. For solutions to quadratic and cubic equations over finite fields of even characteristic, see \cite[Section~2.1]{planesqeven}.

A {\it conic} $\mathcal{C}$ in $\mathrm{PG}(2, q)$ is defined as the zero locus $\mathcal{Z}(f)$ of a quadratic form $f$ on $\mathrm{PG}(2, q)$. Up to projective equivalence, conics in $\mathrm{PG}(2, q)$ fall into four orbits: double lines, pairs of distinct (real) lines, pairs of conjugate (imaginary) lines defined over the quadratic extension $\mathbb{F}_{q^2}$, and non-singular conics.

A non-trivial subspace of the projective space of quadratic forms in $\mathbb{F}_q[X, Y, Z]$ has dimension $1 \leq s \leq 4$, and is called a pencil, net, web, or squab of conics when $s = 1, 2, 3$, or $4$, respectively. We denote by $\mathcal{N} = \langle \mathcal{C}_1, \mathcal{C}_2, \mathcal{C}_3 \rangle$ a net of conics generated by three conics $\mathcal{C}_i$ that do not lie in a pencil of conics. The {\it base} of a net of conics $\mathcal{N}$ is defined as the set of points in the intersection $\mathcal{C}_1 \cap \mathcal{C}_2 \cap \mathcal{C}_3$. 

\subsection{The Veronese surface $\cV(\Fq)$ in $\PG(5,q)$}\label{egSolid}

The {\it Veronese surface} $\cV(\bF_q)$ is a 2-dimensional algebraic variety in $\PG(5,q)$ defined as the image of the {\it Veronese embedding}
\[
\nu: \PG(2,q)\rightarrow \PG(5,q) \quad \text{given by} \quad
(u_0,u_1,u_2) \mapsto (u_0^2,u_0u_1,u_0u_2,u_1^2,u_1u_2, u_2^2).
\]

 A point $P = (y_0, y_1, y_2, y_3, y_4, y_5)$ in $\mathrm{PG}(5, q)$ can be represented by the symmetric $3 \times 3$ matrix
\[
M_P = \begin{bmatrix} 
y_0 & y_1 & y_2 \\
y_1 & y_3 & y_4 \\
y_2 & y_4 & y_5  
\end{bmatrix}.
\]
This matrix representation naturally extends to subspaces of $\mathrm{PG}(5, q)$. For instance, the plane $\pi$ spanned by the first three points of the standard frame in $\mathrm{PG}(5, q)$ is given by
\[
M_\pi = \begin{bmatrix} 
x & y & z \\
y & \cdot & \cdot \\
z & \cdot & \cdot 
\end{bmatrix},
\]
where each “$\cdot$” denotes a zero entry.\\

The {\it rank} of a point $P \in \PG(5, q)$ is defined as the rank of its associated symmetric matrix $M_P$. Points of rank one correspond to points of the Veronese surface $\mathcal{V}(\Fq)$. The {\it secant variety} $\mathcal{V}^{(2)}(\Fq)$ consists of all points of rank at most two in $\PG(5, q)$. Points in $\pi\cap\mathcal{V}^{(2)}(\Fq)$, where $\pi$ is a plane in $\PG(5,q)$, correspond to the $\bF_q$-rational points of the cubic curve  $\mathscr{C}(\pi)$ in $\PG(2,q)$ defined by setting the determinant of the matrix representation of $\pi$ to zero. The surface $\mathcal{V}(\Fq)$ contains $q^2 + q + 1$ {\it conics}, which are the images under the Veronese embedding $\nu$ of lines in $\PG(2, q)$ \cite[Remark 2.1]{webs}. For any two points $P, Q \in \mathcal{V}(\Fq)$, there exists a unique conic through them, defined as
\[
\mathcal{C}(P, Q) := \nu\left(\langle \nu^{-1}(P), \nu^{-1}(Q) \rangle\right),
\]
with any two such conics intersecting in exactly one point. Each conic in $\PG(2, q)$ corresponds to a hyperplane section of $\mathcal{V}(\Fq)$ via the map
\begin{eqnarray}\label{eqn:delta}
\delta: \cZ\left(\sum_{0 \leq i \leq j \leq 2} a_{ij} X_i X_j\right) \mapsto \mathcal{Z}(a_{00}Y_0 + a_{01}Y_1 + a_{02}Y_2 + a_{11}Y_3 + a_{12}Y_4 + a_{22}Y_5).
\end{eqnarray}

A {\it conic plane} in $\PG(5, q)$ is a plane intersecting $\mathcal{V}(\Fq)$ in a conic. For $q$ even, all tangent lines to a conic $\mathcal{C} \subset \mathcal{V}(\Fq)$ are concurrent at its {\it nucleus}. The set of nuclei of all such conics lies in a plane $\pi_{\mathcal{N}} = \mathcal{Z}(Y_0, Y_3, Y_5)$, known as the {\it nucleus plane}. Every point of rank two in $\PG(5, q)$ lies in a unique conic plane $\langle \mathcal{C}(R) \rangle$ containing a conic of $\mathcal{V}(\Fq)$. \\

\subsection{Correspondence and combinatorial invariants}

There exists a one-to-one correspondence between the $\PGL(3, q)$-orbits of linear systems of conics in $\PG(2, q)$ and the $K$-orbits of subspaces in $\PG(5, q)$, where $K$ is a subgroup of $\PGL(6, q)$ isomorphic to $\PGL(3, q)$. The group $K$ arises as the lift of $\PGL(3, q)$ via the Veronese embedding $\nu$. Specifically, $K$ is defined as $K := \alpha(\PGL(3, q))$, where $\alpha: \PGL(3, q) \rightarrow \PGL(6, q)$ is a group monomorphism defined as follows.
For a projectivity $\phi_A \in \PGL(3, q)$ defined by the matrix $A \in \GL(3, q)$, the action of the associated projectivity $\alpha(\phi_A) \in \PGL(6, q)$ on points of $\PG(5, q)$ is defined as
\[
\alpha(\phi_A): P \mapsto Q \quad \text{where} \quad M_Q = A M_P A^T,
\]
where $M_P$ and $M_Q$ are the matrices corresponding to the points $P$ and $Q$. The group $K$ preserves the Veronese surface $\mathcal{V}(\mathbb{F}_q)$ in $\PG(5, q)$. When $q > 2$, $K$ coincides with the full setwise stabiliser of $\mathcal{V}(\mathbb{F}_q)$ in $\PGL(6, q)$. In contrast, for $q = 2$, the full setwise stabiliser of $\mathcal{V}(\mathbb{F}_2)$ is strictly larger and is isomorphic to the symmetric group $\operatorname{Sym}_7$.\\

Specifically, the $K$-orbits of points, lines, planes and solids in $\PG(5,q)$ correspond to the $\PGL(3,q)$-orbits of squabs, webs, nets and pencils of conics in $\PG(2,q)$. \\

There are four $\PGL(3,q)$-orbits of squabs of conics in $\PG(2,q)$ for $q$ even, corresponding to the $K$-orbits of points in $\PG(5,q)$, defined as follows:
\begin{itemize}
    \item $\mathcal{P}_1$: the set of rank-$1$ points, consisting of the $q^2 + q + 1$ points in $\cV(\Fq)$;
    \item $\mathcal{P}_3$: the set of rank-$3$ points, with cardinality $q^5 - q^2$;
    \item $\mathcal{P}_{2,n}$: the $q^2 + q + 1$ rank-$2$ points lying in the nucleus plane $\pi_{\mathcal{N}}$;
    \item $\mathcal{P}_{2,s}$: the $(q^2 - 1)(q^2 + q + 1)$ rank-$2$ points in conic planes not contained in $\pi_{\mathcal{N}} \cup \cV(\mathbb{F}_q)$.
\end{itemize}

The point-orbit distribution of a subspace $W \subset \PG(5,q)$ with $q$ even is given by:
\[
OD_0(W) = [r_1, r_{2,n}, r_{2,s}, r_3],
\]
where $r_1 = r_1(W)$ and $r_3 = r_3(W)$ count the number of rank-$1$ and rank-$3$ points in $W$, respectively; $r_{2,n} = r_{2,n}(W)$ counts rank-$2$ points in $W \cap \pi_{\mathcal{N}}$; and $r_{2,s} = r_{2,s}(W)$ counts rank-$2$ points in $W \setminus (\pi_{\mathcal{N}} \cup \cV(\Fq))$. Note that in $\PG(5,2)$, the point-orbit distribution is not preserved under the action of $\operatorname{Sym}_7$, the full setwise stabiliser of $\cV(\mathbb{F}_2)$, since the nucleus plane $\pi_{\mathcal{N}}$ is not stabilised by $\operatorname{Sym}_7$.\\

The $K$-orbits of lines in $\PG(5,q)$ were classified in \cite{lines}, and the $K$-orbits of solids for even $q$ were determined in \cite{solidsqeven}. The classification of planes intersecting $\cV(\Fq)$ for even $q$ was completed in \cite{planesqeven}. Collectively, these results led to the classification of pencils, webs, and non-empty base nets of conics in $\PG(2,q)$; $q$ even. Note that each rank-$1$ point in a plane $\pi$ in $\PG(5,q)$ corresponds, via the Veronese embedding, to a base point of the associated net of conics \cite[Lemma 4.1]{planesqeven}. 
Hyperplanes in $\PG(5,q)$ are associated with conics in $\PG(2,q)$ through the map $\delta$ defined in \eqref{eqn:delta}. There are four $K$-orbits of hyperplanes, corresponding to the distinct $\PGL(3,q)$-orbits of conics:
\begin{itemize}
    \item $\cH_1$: the set of hyperplanes corresponding to double lines; these hyperplanes intersect $\cV(\Fq)$ in a conic;
    \item $\cH_{2,r}$: the set of hyperplanes corresponding to pairs of real lines; these hyperplanes intersect $\cV(\Fq)$ in two conics;
    \item $\cH_{2,i}$: the set of hyperplanes corresponding to conjugate pairs of imaginary lines; these intersect $\cV(\Fq)$ in a single point;
    \item $\cH_3$: the set of hyperplanes corresponding to non-singular conics; these intersect $\cV(\Fq)$ in a normal rational curve of degree $4$.
\end{itemize}

The hyperplane-orbit distribution of a subspace $W \subset \PG(5,q)$ is given by:
\[
OD_4(W) = [h_1, h_{2,r}, h_{2,i}, h_3],
\]
where each $h_i = h_i(W)$ counts the number of hyperplanes from the corresponding $K$-orbit that are incident with $W$. This distribution is often referred to as the {\it conic distribution} of the linear system of conics associated with $W$. Similarly, one can define the {\it line-, plane-}, and {\it solid-orbit distributions} of a subspace $W \subset \PG(5,q)$. These distributions remain invariant under the action of $K$ and play a crucial role in identifying the $K$-orbit of $W$.\\

We conclude this section by presenting the following result from \cite{CSRD codes}, which specifies how the number of hyperplanes in $\mathcal{H}_1$ containing a plane $\pi \subset \mathrm{PG}(5, q)$ depends on the value of $\lvert \pi \cap \pi_{\mathcal{N}} \rvert$.

\begin{Theorem}\cite[Theorem 3.7]{CSRD codes}\label{h1=r2n}
    For each plane $\pi$ in $\PG(5,q)$, $q\geq 4$ even, we have $r_{2,n}(\pi)=h_1(\pi)$. 
\end{Theorem}

\section{Planes intersecting $\cV(\Fq)$ and $\pi_{\mathcal{N}}$}\label{main1}

The $K$-orbits of planes that intersect both the Veronese surface $\cV(\Fq)$ and the nucleus plane $\pi_{\mathcal{N}}$ non-trivially are precisely the orbits $\Sigma_1$, $\Sigma_3$, $\Sigma_4$, $\Sigma_7$, $\Sigma_8$, $\Sigma_9$, $\Sigma_{10}$, $\Sigma_{11}$, and $\Sigma_{15}$, as listed in Table~\ref{tableplanesintersectingpiN}. These orbits were previously classified and characterized in \cite{planesqeven}, where the existence of $15$ $K$-orbits of planes intersecting $\cV(\Fq)$ non-trivially was established; among these, $9$ also have a nonempty intersection with $\pi_{\mathcal{N}}$. For further details on the derivation of these orbits, we refer the reader to \cite[Section 3]{planesqeven}.

\begin{table}[h]
\begin{center}
\footnotesize{

\begin{tabular}[!htbp]{ccc cccc} 
 \toprule
$\pi^K$ & Representatives   & $OD_0(\pi)$ &$\pi^K$ & Representatives   & $OD_0(\pi)$\\ \midrule
\vspace{0.2cm}

$ \Sigma_1$& $\begin{bmatrix} x&y&.\\y&z&.\\.&.&.          \end{bmatrix}$ & $[q+1,1,q^2-1,0]$  &$ \Sigma_{\mathcal{N}}$& $\begin{bmatrix} .&x&y\\x&.&z\\y&z&.          \end{bmatrix}$ & $[0,q^2+q+1,0,0]$  
 \\ \vspace{0.2cm}

$\Sigma_{3}$  &$\begin{bmatrix} x&.&z\\.&y&.\\z&.&.        \end{bmatrix}$ &$[2,1,2q-2,q^2-q]$ &$\Sigma_{16}$  &$\begin{bmatrix}\cdot&x&z\\x&z&y\\ z&y&\cdot\end{bmatrix}$ &$[0,q+1,0,q^2]$ 
 \\\vspace{0.2cm}
 
$\Sigma_{4}$  &$\begin{bmatrix} x&.&z\\.&y&z\\z&z&.        \end{bmatrix}$ &$[2,1,2q-2,q^2-q]$ &  $ \Sigma_{17}$&	$\begin{bmatrix}\cdot&x&y\\x&z&\cdot\\ y&\cdot &z\end{bmatrix}$ & $[0,q+1,q,q^2-q]$
   \\ \vspace{0.2cm}

  $ \Sigma_7$&	$\begin{bmatrix} x&y&z\\y&.&.\\z&.&.          \end{bmatrix}$ & $[1,q+1,q^2-1,0]$& $\Sigma_{18}$ &$\begin{bmatrix}x&y&z\\y&cz&x+z\\ z&x+z&\cdot\end{bmatrix}$ &$[0,1,0,q^2+q]$ 
\\ \vspace{0.2cm}
   
   $ \Sigma_8$&	$\begin{bmatrix} x&y&.\\y&.&z\\.&z&.          \end{bmatrix}$ & $[1,q+1,q-1,q^2-q]$ &$\Sigma_{19}$ & $\begin{bmatrix}x&y&.\\y&y+z&z\\ .&z&x\end{bmatrix}$ &$[0,1,3q,q^2-2q]$
   \\ \vspace{0.2cm}
   
   $ \Sigma_9$&	$\begin{bmatrix} x&y&.\\y&z&z\\.&z&.          \end{bmatrix}$ & $[1,1,2q-1,q^2-q]$&  $ \Sigma_{20}$& $\begin{bmatrix}x&y&bx\\y&cx+y+z&z\\ bx&z&x\end{bmatrix}$	
     & $[0,1,q,q^2]$
   \\ \vspace{0.2cm}

 $\Sigma_{10}$ & $\begin{bmatrix} x&y&.\\y&z&.\\.&.&z          \end{bmatrix}$ &$[1,1,2q-1,q^2-q]$ & $ \Sigma_{21}$&	$\begin{bmatrix} x & x+az& \cdot \\ x+az  &z  & y\\ \cdot & y & \cdot \end{bmatrix}$ & $[0,1,2q,q^2-q]$ 
 \\\vspace{0.2cm}

    $ \Sigma_{11}$& $\begin{bmatrix} x&y&.\\y&z&z\\.&z&x+z         \end{bmatrix}$	
     & $[1,1,q-1,q^2]$&  $ \Sigma_{22}$&	$\begin{bmatrix} x & x+z& z \\ x+z  &z  & y\\ z & y & \cdot \end{bmatrix}$ & $[0,1,q,q^2] $
     \\ \vspace{0.2cm}

    $ \Sigma_{15}$&	$\begin{bmatrix} x&y&z\\y&z&.\\z&.&.         \end{bmatrix}$ & $[1,1,q-1,q^2]$& $ \Sigma_{23}$&	$\begin{bmatrix} x & az& x \\ az  &z  & y\\ x & y & \cdot \end{bmatrix}$ & $[0,1,2q,q^2-q]$
    \\  \bottomrule

   \end{tabular}}

 \caption{\label{tableplanesintersectingpiN}$K$-orbits of planes in $\PG(5,q)$ intersecting  $\pi_{\mathcal{N}}$ in at least one point and their point-orbit distributions for $q> 2$ even. The parameter $c$ in $\Sigma_{18}$ is a non-admissible element in $\Fq$ (see \cite[Section~2.1]{planesqeven}) such that $\operatorname{Tr}(c^{-1})=\operatorname{Tr}(1)$. The parameters in $\Sigma_{20}$ satisfy $b\neq 1$ and $\operatorname{Tr}\left(c/(1 + b^2)\right)=1$. The parameter $a$ in $\Sigma_{21}$ and $\Sigma_{23}$ satisfies $\operatorname{Tr}(a)=1$.}

\end{center}
\end{table}

\section{Planes disjoint from $\cV(\Fq)$ and intersecting  $\pi_{\mathcal{N}}$}\label{main2}

In this section, we classify the $K$-orbits of planes in $\PG(5,q)$ that contain no rank-$1$ points and at least one rank-$2$ point in $\pi_{\mathcal{N}}$. We refer the reader to \cite[Tables~1 and~3]{CSRD codes} for the notation, representation and properties of the $15$ $K$-orbits of lines classified in \cite{lines}.

The following lemma will be used in the proof of the first main result, Theorem \ref{thm:line_in_nucleus_plane}.

\begin{Lemma}\label{lem:$K(o_{12,3})$-action}
Let $\ell$ be a line of type $o_{12,3}$ and $R \in \ell\cap \cP_{2,s}$. The stabiliser $K_\ell$ of $\ell$ in $K$ has three orbits on lines in $\langle \cC(R)\rangle$ through $R$: (i) the tangent line to $\cC(R)$ through $R$, (ii) the secant lines to $\cC(R)$ through $R$, and (iii) the external lines to $\cC(R)$ through $R$.
\end{Lemma}
\begin{proof}
Recall that $OD_0(o_{12,3})=[0,1,q,0]$. 
By \cite{lines}, the $K$-orbit $o_{12,3}$ can be represented by the line $\ell$ with
$$
M_\ell=\begin{bmatrix} \cdot & x& \cdot \\x  &x+y & y\\ \cdot & y & \cdot\end{bmatrix}.
$$
The stabiliser $K_\ell$  in $K$ of a line of $\ell$ of type $o_{12,3}$ was determined in \cite{lines}, where it was shown that $K_\ell$ is isomorphic to $E_q^2 :E_q :C_{q-1}$, and a general element of $K_l$ corresponds to the projectivity $\varphi(d_{11},d_{21},d_{22},d_{23},d_{33})\in \PGL(3,q)$ with matrix
\[
D = \begin{bmatrix}
d_{11} & \cdot & d_{22} + d_{33} \\
d_{21} & d_{22} & d_{23} \\
d_{11} + d_{22} & \cdot & d_{33}
\end{bmatrix}
\]
using the notation from \cite[Section 4]{lines}. Let $R_{x,y}$ denote the point with coordinates $(0,x,0,x+y,y,0)$ on $\ell$, and 
let $P=R_{1,1}$ denote the unique point on $\ell$ contained in the nucleus plane. Let $\ell(P)$ denote the line in $\PG(2,q)$ corresponding to the pre-image of the conic $\cC(P)$ under the Veronese map. Similarly, denote by $\ell(R_{x,y})$ the line in $\PG(2,q)$ corresponding to the conic $\cC(R_{x,y})$.
Then $\ell(P)=\cZ(X_0+X_2)$ and $\ell(R_{x,y})=\cZ(yX_0+xX_2)$, and these $q+1$ lines form a pencil of lines in $\PG(2,q)$ with base $(0,1,0)$. Clearly each $\varphi(d_{11},d_{21},d_{22},d_{23},d_{33})$ fixes the line $\ell(P)$ and the point $(0,1,0)$, and since $\varphi(d_{11},d_{21},d_{22},d_{23},d_{33})$ maps the point $(1,0,0)$ to the point $(d_{11},d_{21},d_{11}+d_{22})$, it follows that $K_\ell$ acts transitively on the conics $\cC(R_{x,y})$ with $x\neq y$, and therefore also on $R \in \ell\cap \cP_{2,s}$. Put $R=R_{1,0}$, and consider the stabiliser $K_{\ell,R}$ in $K_\ell$ of $R$. The elements of $K_{\ell,R}$ correspond to the projectivities
$\varphi(1,d_{21},1,d_{23},d_{33})$ with $d_{33}\neq 0$, and one easily verifies that the group $K_{\ell,R}$ fixes the point $(0,0,0,1,0,0)$ and acts transitively on the remaining points of the conic $\cC(R)$. It follows that $K_{\ell,R}$ fixes the tangent line through $R$, and acts transitively on the secant lines in $\langle \cC(R)\rangle$ through $R$. To see that the $K_{\ell,R}$ also acts transitively on the external lines through $R$ in $\langle \cC(R)\rangle$, it suffices to go to the quadratic extension $\PG(5,q^2)$ and observe that the group
$K_{\ell,R}$ acts transitively on the set of points of $\cC(R)(q^2)$ which are not defined over $\bF_q$.
\end{proof}

\begin{Theorem}\label{thm:line_in_nucleus_plane}
There are two $K$-orbits, $\Sigma_{16}$ and $\Sigma_{17}$, of planes $\pi$ in $\PG(5,q)$ which are disjoint from the Veronese surface $\cV(\bF_q)$, and meet the nucleus plane of $\cV(\bF_q)$ in a line. These two $K$-orbits are characterised by their point-orbit distributions $OD_0(\Sigma_{16})=[0,q+1,0,q^2] $ and $OD_0(\Sigma_{17})=[0,q+1,q,q^2-q]$.
\end{Theorem}
\begin{proof}
 If $\pi$ is a plane with $OD_0(\pi)=[0,q+1,0,q^2]$ then $\pi$ is contained in $q+1$ hyperplanes of the $K$-orbit $\cH_1$, namely the hyperplanes containing the solid $\langle \pi,\pi_\cN\rangle$. Without loss of generality, we may assume that $\pi\in H_0=\cZ(Y_0)\in \cH_1$. Let $\cC_0$ denote the conic of $\cV(\bF_q)$ contained in $H_0$. Since the conic plane $\langle \cC_0\rangle$ and the plane $\pi$ meet in a point of rank two, and $\pi$ has no points of rank two outside of the nucleus plane, it follows that $\pi$ meets $\langle \cC_0\rangle$ in the nucleus $P_0(0,0,0,0,1,0)$ of $\cC_0$. Since the stabiliser of $H_0$ in $K$ acts transitively on the lines of $\pi_\cN$ through $P_0$, we may assume 
 that the line $\pi\cap \pi_\cN$ is the line through $P_0$ and the point with coordinates $(0,1,0,0,0,0)$, and that the plane $\pi$ has the form
$$
\pi_{a,b,c}=\langle \ell_{12,1},P(0,0,a,b,0,c)\rangle
$$
represented by the matrix
$$
\pi_{a,b,c}~:~\begin{bmatrix}\cdot&x&az\\x&bz&y\\ az&y&cz\end{bmatrix},
$$
where $\ell_{12,1}$ is the representative of the $K$-orbit $o_{12,1}$ from \cite{lines}.
By computing the cubic curve $\pi\cap \cV^{(2)}(\Fq)$, one can verify that if $c\neq 0$, then $\pi_{a,b,c}$ contains points of rank $\leq 2$ outside the nucleus plane. So we may put $c=0$. If $a=0$, then $\pi_{a,b,c}$ is contained in $\cV^{(2)}(\Fq)$. So we may put $a=1$, since the set of planes $\pi_{a,b,c}$ is parameterized by points $(a,b,c)$ in $\PG(2,q)$. It is straightforward to show that $\pi_{1,b,0}$ is $K$-equivalent to $\pi_{1,1,0}$. This shows that there is a unique $K$-orbit of planes with point orbit distribution $[0,q+1,0,q^2]$. A representative for this orbit is given by
\begin{eqnarray}\label{eqn:Sigma_16}
\Sigma_{16}~:~\begin{bmatrix}\cdot&x&z\\x&z&y\\ z&y&\cdot\end{bmatrix},
\end{eqnarray}
whose associated cubic curve $\Delta(\pi_{1,1,0})$ is a triple line in $o_{12,1}$.

Next, suppose $\pi$ is a plane which intersects the nucleus plane in a line and contains at least one point $P$ of rank two outside the nucleus plane. Since the stabiliser of the nucleus plane acts transitively on the set of hyperplanes in $\cH_1$, we may assume that $P$ is contained in the conic plane $\langle \cC_0\rangle$ in $H_0=\cZ(Y_0)\in \cH_1$. Note that $\pi$ does not contain the nucleus $P_0$ of $\cC_0$, since otherwise $\pi$ would contain the point of the $\cC_0$ on the tangent line $\langle P,P_0\rangle$. Since the stabiliser of $H_0$ in $K$ acts transitively on lines of $\pi_\cN$ not through $P_0$, we assume that $\pi$ contains the points $(0,0,1,0,0,0)$ and $(0,0,0,0,1,0)$. 
Therefore, the plane $\pi$ has the form
$$
\pi_{a,b,c}~:~\begin{bmatrix}\cdot&x &y\\ x &az&bz\\ y&bz&cz\end{bmatrix}.
$$

These planes are $K$-equivalent to $\pi_{1,0,1}$, which has point-orbit distribution $[0,q+1,q,q^2-q]$.
To see this, first observe that the choice of the line $L$ of $\pi\cap \pi_\cN$ not through $P_0$, corresponds to fixing a pencil of lines in the pre-image $\PG(2,q)$ of the Veronese map, consisting of the lines corresponding to the conics on the Veronese variety with nucleus on the line $L$. With the choices that were made above, this corresponds to the pencil of lines in $\PG(2,q)$ with base $(1,0,0)$. So the stabiliser $K_{\mathcal F}$ in $K$ of the flag ${\mathcal F}=(L,H_0)$ is isomorphic to the group of homologies with center $(1,0,0)$ and axis $\cZ(X_0)$. Let $Q$ denote the point of $\cC_0$ on the tangent line $\langle P,P_0\rangle$, and $R$ and $S$ two points on a secant line of $\cC_0$ through $P$. Then $P$ is determined by the triple $(Q,R,S)$ of distinct points on $\cC_0$ and $K_{\mathcal F}$ acts transitively on such triples.

 This shows that there is a unique $K$-orbit of planes with point orbit distribution $[0,q+1,q,q^2-q]$. A representative for this orbit is given by
\begin{eqnarray}\label{eqn:Sigma_18}
\Sigma_{17}~:~\begin{bmatrix}\cdot&x&y\\x&z&\cdot\\ y&\cdot &z\end{bmatrix},
\end{eqnarray}
whose associated cubic curve $\Delta(\pi_{1,0,1})$ is the union of a line in $o_{12,1}$ and a double line in $o_{12,3}$. This completes the proof.
\end{proof}

The next objective is to classify planes which are disjoint from the Veronese surface, and which meet the nucleus plane in a point. The following lemma will be used in the proof of the second main result of this paper, Theorem \ref{thm:point_in_nucleus_plane}.

\begin{Lemma}\label{lem:lines_o_{13,1}}
Let $P$ be a point in the nucleus plane of the Veronese surface $\cV(\bF_q)$, and let $H(P)$ denote the unique hyperplane of $\cH_1$ containing the conic $\cC(P)$. For any $H\in \cH_1\setminus\{H(P)\}$, there are two $K$-orbits of lines of type $o_{13,1}$ through $P$ in $H$ not in $H(P)$ under the stabiliser $K_{P,H}$ of $P$ and $H$ in $K$. 
\end{Lemma}
\begin{proof}
Let $\ell=\langle P, R\rangle$, be a line of type $o_{13,1}$ where $R$ is the unique point of rank two on $\ell\setminus \{P\}$.
Without loss of generality we may assume that $H(P)=\cZ(Y_0)$ and $H=\cZ(Y_5)$.
Let $\ell(R)$ denote the pre-image of the conic $\cC(R)$ under the Veronese map, and likewise denote by $\ell(P)$ and $\ell(H)$, the lines of $\PG(2,q)$ determined by the two hyperplanes $H(P),H\in \cH_1$.
If the three lines $\ell(R)$, $\ell(P)$, and $\ell(H)$ are not pairwise distinct, then $\ell(R)=\ell(H)$, since $R\notin H(P)$, and
the line $\ell$ can be represented by
$$
M_\ell=\begin{bmatrix} x & \beta x& \cdot \\ \beta x  &\gamma x & y\\ \cdot & y & \cdot\end{bmatrix},
$$
for some $\beta,\gamma\in \bF_q$. Under the stabiliser  $K_{P,H}$ the corresponding line $\ell$ is equivalent to the line represented by the matrix
$$
\begin{bmatrix} x & x& \cdot \\ x  &\cdot & y\\ \cdot & y & \cdot\end{bmatrix}. 
$$

\bigskip

If the three lines $\ell(R)$, $\ell(P)$, and $\ell(H)$, are pairwise distinct and concurrent, then by transitivity, we may assume that $\ell(R)$ has equation $\cZ(X_0+X_2)$.
In this case the conic plane $\langle \cC(R)\rangle$ corresponds to the matrix
$$
\begin{bmatrix} y+z & z& y+x \\ z  & x+z& z\\ y+z& z & y+z\end{bmatrix},
$$
and therefore meets both hyperplanes $H(P)$ and $H$ in the same line.
This is a contradiction since $R\in \langle \cC(R)\rangle$, $R\in H$, but $R\notin H(P)$.

\bigskip

If the three lines $\ell(R)$, $\ell(P)$, and $\ell(H)$, are pairwise distinct but not concurrent, then without loss of generality we may assume that $\ell(R)=\cZ(X_1)$, and therefore the plane $\langle \cC(R)\rangle=\cZ(Y_1,Y_3,Y_4)$. Since $R\in H\setminus H(P)$, $R$ has coordinates $(1,0,\beta,0,0,0,)$, for some $\beta \in \bF_q$ and $\ell=\langle P,R\rangle$ has matrix
$$
M_\ell=
\begin{bmatrix} x & \cdot & \beta x \\ \cdot  & \cdot & y\\   \beta x & y & \cdot \end{bmatrix},
$$
and is therefore in the same $K_{P,H}$-orbit as the line represented by the matrix 
$$
\begin{bmatrix} x & \cdot& x \\ \cdot  &\cdot  & y\\ x & y & \cdot \end{bmatrix}.
$$
This completes the proof.
\end{proof}

Let $P$ be a point in the nucleus plane $\pi_\cN$ of $\cV(\bF_q)$. As before, let $H(P)$ denote the unique hyperplane in $\cH_1$ containing $\cC(P)$. If $\pi$ is a plane disjoint from the Veronese variety, which meets the nucleus plane in the point $P$, then by Theorem \ref{h1=r2n}, there is a unique hyperplane in $\cH_1$ containing $\pi$. Denote this hyperplane by $H(\pi)$.

\begin{Theorem}\label{thm:same_hyperplanes}
If $\pi$ is a plane in $\PG(5,q)$ which is disjoint from the Veronese surface $\cV(\bF_q)$, and meets the nucleus plane of $\cV(\bF_q)$ in a point $P$, with $H(P)=H(\pi)$, then $\pi$ belongs to one of three $K$-orbits $\Sigma_{18}$, $\Sigma_{19}$, and $\Sigma_{20}$, with
$OD_0(\Sigma_{18})=[0,1,0,q^2+q]$, $OD_0(\Sigma_{19})=[0,1,3q,q^2-2q]$, and $OD_0(\Sigma_{20})=[0,1,q,q^2]$.
\end{Theorem}
\begin{proof}
Let $\pi$ be a plane in $\PG(5, q)$ with
$OD_0(\pi) = [0, 1, r_{2s} > 1, r_3]$,
meeting $\pi_\cN$ in the point $P$, such that $H(P)=H(\pi)$.
Note that, by~\cite{webs,lines}, if $\ell$ is a line such that $P \notin \ell$ and $\pi = \langle P, \ell \rangle$, then $\ell$ must lie in $o_{10} \cup o_{13,3} \cup o_{14,1} \cup o_{15,1} \cup o_{16,3} \cup o_{17}$, as these are the only line types that are contained in a unique hyperplane in $\cH_1$, and are disjoint from $\cV(\Fq) \cup \pi_{\mathcal{N}}$.\\
Since $H(P)=H(\pi)$, $\ell \in H(P)$, and therefore $\ell \in o_{14,1}\cup o_{15,1}\cup o_{17}$. This follows from the fact that the unique hyperplane in $\cH_1$ containing a line $\ell \in o_{10} \cup o_{13,3} \cup o_{16,3}$ intersects $\cV(\Fq)$ in  $\mathcal{C}(Q)$, where $Q \in \ell\cap \mathcal{P}_{2s}$, and this would imply that the line $\langle P,Q\rangle$ is a tangent line to the conic $\cC(Q)=\cC(P)$, a contradiction. \\
Let $K_P$ and $K_{\ell}$ denote the stabilisers of $P$ and $\ell$ in $K$, respectively. Note that $K_P=K_{H(P)}$.  Since a line $\ell \in o_{14,1}\cup o_{15,1}\cup o_{17}$ is contained in a unique hyperplane in $\cH_1$ (by \cite[Lemma 4.26]{webs}), it follows that $K_{\ell}\subset K_P$ for each line $\ell$ in $o_{14,1}\cup o_{15,1}\cup o_{17}$ contained in $H(P)$.\\

Assume first that $\ell=\ell_{17} \in o_{17}$. By \cite[Theorem 3.9]{CSRD codes}, this returns us to the unique orbit $\Sigma_{18}$ with point-orbit distribution $[0,1,0,q^2+q]$. Let $P$ be the point parametrized by $(x,y,z)=(0,1,0)$. A representative of $\Sigma_{18}$ can be obtained by considering $\langle P,\ell_{17}\rangle \subset H(P)=\mathcal{Z}(Y_5)$.  Let $c\in \Fq$ be a non-admissible element (see \cite[Section~2.1]{planesqeven}) such that $\operatorname{Tr}(c^{-1})=\operatorname{Tr}(1)$, and consider $\ell_{17}=\langle (1,0,0,0,1,0),(0,0,1,c,1,0) \rangle \in o_{17}$ (see Section \ref{pre}). Then,  $\ell_{17}\subset H(P)$, and  a representative of  $\Sigma_{18}$ is given by 
    
\begin{eqnarray}\label{eqn:Sigma_17}
\Sigma_{18}~:~ \begin{bmatrix}x&y&z\\y&cz&x+z\\ z&x+z&\cdot\end{bmatrix},
\end{eqnarray}

and its associated cubic curve is a point  defined by $\mathcal{C}_{18}=\mathcal{Z}(X^3 + XZ^2 + cZ^3)$.\\

Assume next that $\ell=\ell_{14,1} \in o_{14,1}$.  By \cite{lines}, $|K_{\ell_{14,1}}|=6$. Therefore the orbit of $\ell_{14,1}$ under $K_P$ has size
$$
\frac{|K_P|}{|K_{\ell_{14,1}}|}=\frac{q^2|\GL(2,q)|}{6}=\frac{1}{6}q^3(q-1)(q^2-1).
$$
By \cite[Theorem 4.29]{webs} this is equal to the total number of lines in $o_{14,1}$ contained in $H(P)$. Hence, $K_P$ acts transitively on the lines in $H(P)$ contained in $o_{14,1}$, and thus there exists a unique such $K$-orbit in $\PG(5,q)$. We call this orbit $\Sigma_{19}$. A representative of $\Sigma_{19}$ can be obtained by considering $\langle P, \ell_{14,1}\rangle$ where $P=(0,1,0,0,1,0)\in \pi_{\mathcal{N}}$ and $\ell_{14,1}=\langle (1,0,0,0,0,1),(0,1,0,1,0,0)\rangle$. In this case, $H(P)=\mathcal{Z}(Y_0+Y_5)$ and a representative of $\Sigma_{19}$ is given by 

 \begin{eqnarray}\label{eqn:Sigma_19}
 \Sigma_{19}~:~ \begin{bmatrix}x&y&.\\y&y+z&z\\ .&z&x\end{bmatrix}.
 \end{eqnarray}
 
The cubic curve associated with $\Sigma_{19}$ is $\mathcal{C}_{19} = \mathcal{Z}(X(Y+Z)(X+Y+Z))$, which defines three lines in $o_{12,3}$ passing through the point $P$. Consequently, the point-orbit distribution of $\pi \in \Sigma_{19}$ is given by $OD_{0}(\pi) = [0,1,3q,q^2 - 2q]$. The other $q - 2$ lines in $\pi$ through $P$ are of type $o_{16,1}$, while the remaining $q^2$ lines in $\pi$ are of type $o_{14,1}$. \\

Finally, let $\ell=\ell_{15,1} \in o_{15,1}$. By \cite{lines}, $|K_{\ell_{15,1}}|=2$. Therefore the orbit of $\ell_{15,1}$ under $K_P$ has size
$$
\frac{|K_P|}{|K_{\ell_{15,1}}|}=\frac{q^2|\GL(2,q)|}{6}=\frac{1}{2}q^3(q-1)(q^2-1).
$$
By \cite[Theorem 4.29]{webs} this is equal to the total number of lines in $o_{15,1}$ contained in $H(P)$.
Hence $K_P$ acts transitively on the lines in $H(P)$ which belong to $o_{15,1}$, and thus there exists a unique such $K$-orbit in $\PG(5,q)$. This defines a new orbit $\Sigma_{20}$. A representative can be obtained by considering $\langle P, \ell_{15,1}\rangle$ where $P=(0,1,0,0,1,0)\in \pi_{\mathcal{N}}$ and $\ell_{15,1}=\langle (1,0,b,c,0,1),(0,1,0,1,0,0)\rangle$; $b\neq 1$ and $\operatorname{Tr}\left(c/(1 + b^2)\right)=1$. In this case, $H(P)=\mathcal{Z}(Y_0+Y_5)$ and $\Sigma_{20}$ can be represented by 
 \begin{eqnarray}\label{eqn:Sigma_20}
\Sigma_{20}:  \begin{bmatrix}x&y&bx\\y&cx+y+z&z\\ bx&z&x\end{bmatrix}. 
\end{eqnarray}
 The cubic curve associated with $\Sigma_{20}$ is $\mathcal{C}_{20} = \mathcal{Z}(X(c(1+b^2)X^2+(1+b^2)X(Y+Z)+(Y+Z)^2))$, which defines a line in $o_{12,3}$ and a pair of imaginary lines  meeting at $P$. Consequently, the point-orbit distribution of $\pi \in \Sigma_{20}$ is given by $OD_{0}(\pi) = [0,1,q,q^2]$. 
\end{proof}

\begin{Theorem}\label{thm:different_hyperplanes}
If $\pi$ is a plane in $\PG(5,q)$ which is disjoint from the Veronese surface $\cV(\bF_q)$, and meets the nucleus plane of $\cV(\bF_q)$ in a point $P$, with $H(P)\neq H(\pi)$, then $\pi$ belongs to one of three $K$-orbits $\Sigma_{21}$, $\Sigma_{22}$, and $\Sigma_{23}$, with $OD_0(\Sigma_{21})=[0,1,2q,q^2-q]$, $OD_0(\Sigma_{22})=[0,1,q,q^2]$ and $OD_0(\Sigma_{23})=[0,1,2q,q^2-q]$.

\end{Theorem}

\begin{proof}
The fact that there is one $K$-orbit of planes $\pi$ with point-orbit distribution $OD_0(\pi)=[0,1,0,q^2+q] $ was proved in \cite{CSRD codes}, where it was also shown that such a plane has hyperplane-orbit distribution $[1,0,0,q^2+q]$. This corresponds to the $K$-orbit $\Sigma_{18}$, in which case $H(P)=H(\pi)$, see Theorem \ref{thm:same_hyperplanes}.
Hence, since $H(P)\neq H(\pi)$, the plane $\pi$ contains at least one point $R$ of rank 2 outside $\pi_{\mathcal{N}}$. 

Let $\ell$ be a line of $\pi$ through $P$, not contained in $H(P)$, and $Q\in H(P)$ be such that $\pi=\langle \ell, Q\rangle$.

Since $\ell$ is a line disjoint from the Veronese surface which meets the nucleus plane in a point, by \cite{lines}, $\ell$ must belong to one of the $K$-orbits $o_{12,3}$, $o_{13,1}$ or $o_{16,1}$. 

By \cite{lines}, the $K$-orbit $o_{12,3}$ can be represented by the line $\ell$ with
$$
M_\ell=\begin{bmatrix} \cdot & x& \cdot \\x  &x+y & y\\ \cdot & y & \cdot\end{bmatrix}.
$$
This line $\ell$ intersects the nucleus plane in the point $P$ with coordinates $(0,1,0,0,1,0)$. The corresponding conic plane $\rho$ containing $\cC(P)$ has matrix
$$
M_\rho=\begin{bmatrix} x & y& x \\y  &z & y\\ x & y & x\end{bmatrix},
$$
and therefore the hyperplane $H(P)=\cZ(Y_0+Y_5)$. It follows that $\ell$ is contained in $H(P)$, a contradiction.

Similarly, one easily verifies that if $\ell\in o_{16,1}$, and $P=\ell \cap \pi_\cN$, then $\ell \subseteq H(P)$, again contradicting the hypothesis. To see this, it suffices to consider the representative for $\ell$ with matrix
$$
M_\ell=\begin{bmatrix} \cdot & \cdot& x \\ \cdot  &x & y\\ x & y & \cdot\end{bmatrix},
$$
from \cite{lines}.

It follows that $\ell$ must belong to the $K$-orbit $o_{13,1}$, which has $OD_0(\ell)=[0,1,1,q-1]$.

\bigskip

Since $P$ is uniquely determined by the conic $\cC(P)$ and $H(\pi)$ is uniquely determined by the conic $\cC(\pi)$, the fact that $\PGL(3,q)$ acts transitively on pairs of lines in $\PG(2,q)$ implies that without loss of generality we may assume that $P$ has coordinates $(0,0,0,0,1,0)$, $H(P)=\cZ(Y_0)$ and $H(\pi)=\cZ(Y_5)$.

With the notation as above, we may assume that $R$ is the unique point of rank two such that $\ell=\langle P,R\rangle$. Consider the three conic planes determined by $\cC(P)$, $\cC(\pi)$ and $\cC(R)$. By construction $\cC(P)\neq \cC(\pi)$ and $\cC(P)\neq \cC(R)$. By the above choice of coordinates the point $\cC(P)\cap \cC(\pi)$ has coordinates $(0,0,0,1,0,0)$.

By Lemma \ref{lem:lines_o_{13,1}}, there are two orbits of lines $\ell=\langle P, R\rangle$, where $R$ is the unique point of rank two on $\ell\setminus \{P\}$, of type $o_{13,1}$ through $P$ in $H(\pi)$ not in $H(P)$ under the stabiliser $K_P$ of $P$, and they can be represented by (see proof of Lemma \ref{lem:lines_o_{13,1}})
$$
\begin{bmatrix} x & x& \cdot \\x  &\cdot  & y\\ \cdot & y & \cdot \end{bmatrix} \mbox{,  and }
\begin{bmatrix} x & \cdot& x \\ \cdot  &\cdot  & y\\ x & y & \cdot \end{bmatrix}.
$$
The first has $H(R)=H(\pi)$ and the second has $H(R)\neq H(\pi)$, where $R$ is the point corresponding to the matrix obtained by setting $(x,y)=(1,0)$.

\bigskip

Choosing a point $Q$ in $H(P)\cap H(\pi)$ such that $\pi=\langle \ell, Q\rangle$, we obtain the following two sets, say $\mathcal S$ and $\mathcal T$, of planes represented by
$$
\sigma_{a,b,c}~:~\begin{bmatrix} x & x+az& bz \\ x+az  &cz  & y\\ bz & y & \cdot \end{bmatrix}
$$
and
$$
\tau_{a,b,c}~:~\begin{bmatrix} x & az& x+bz \\ az  &cz  & y\\ x+bz & y & \cdot \end{bmatrix},
$$
respectively, where each of the sets $\mathcal S$ and $\mathcal T$ consists of $q^2+q+1$ planes parametrised by $(a,b,c) \in \PG(2,q)$.
Note that $c$ must be nonzero, since we assume $\pi$ meets the nucleus plane in the point $P$. Putting $c=1$, we are left with $q^2$ planes in each of the sets $\mathcal S$ and $\mathcal T$.

\bigskip

\underline{The set $\mathcal S$.}
The cubic ${\mathcal{K}}(\sigma_{a,b,1})$ obtained as the intersection $\sigma_{a,b,1} \cap \cV^{(2)}(\bF_q)$ is the zero locus of the form
$$
bZ((X+aZ)Y+bZ^2)+Y(XY+(X+aZ)bZ)=b^2Z^3+XY^2.
$$
For $b=0$ the cubic ${\mathcal{K}}(\sigma_{a,0,1})$ is the union of two lines, one of which is double, and $OD_0(\sigma_{a,0,1})=[0,1,2q,q^2-q]$. The plane $\sigma_{a,0,1}$ meets the conic plane $\langle \cC(\sigma_{a,0,1})\rangle$ in a line $m$ with matrix
$$
M_m=\begin{bmatrix} x & x+az& \cdot \\ x+az  &z  & \cdot\\ \cdot & \cdot & \cdot \end{bmatrix}
$$
which must be external to $\cC(\sigma_{a,0,1})$ ($m\in o_{10}$), implying that $\operatorname{Tr}(a)=1$. Note that the line in $\sigma_{a,0,1}$ containing the points parametrised by $(x,y,z)=(0,y,z)$ is a line through $P$ belonging to $o_{12,3}$.
To prove that each two planes $\sigma_{a,0,1}$ with $\operatorname{Tr}(a)=1$ belong to the same $K$-orbit it suffices to apply Lemma \ref{lem:$K(o_{12,3})$-action}.
We denote this orbit by $\Sigma_{21}$, and it can be represented by
$$
\Sigma_{21}~:~\begin{bmatrix} x & x+az& \cdot \\ x+az  &z  & y\\ \cdot & y & \cdot \end{bmatrix};
$$
$\operatorname{Tr}(a)=1$.\\

If $b\neq 0$ then the cubic ${\mathcal{K}}(\sigma_{a,b,1})$  is irreducible and consists of the points parameterised by $(x,y,z)$ belonging to the set
$$\{(1,0,0),(0,1,0)\} \cup \{(b^2 y^2,y,1): y \in \bF_q\setminus \{0\}\},$$
and $OD_0(\sigma_{a,b,1})=[0,1,q,q^2]$. The plane $\sigma_{a,b,1}$ meets the conic plane $\langle \cC(\pi)\rangle$ in the point $R$, and $R$ is a double point of ${\mathcal{K}}(\sigma_{a,b,1})$.

First, observe that the planes $\sigma_{a,b,1}$ and $\sigma_{1,b,1}$ are $K$-equivalent. This leaves us with $q-1$ planes $\sigma_{1,b,1}$ with $b\in \bF_q\setminus \{0\}$. The plane $\sigma_{1,b,1}$ is spanned by the points $P$, $R$, and $Q_{b,y}$ of ${\mathcal{K}}(\sigma_{1,b,1})$, where $Q_{b,y}$ is parametrised by $(x,y,z)=(b^2/y^2,y,1)$, $y\in \bF_q\setminus \{0\}$. The conic $\cC(Q_{b,y})$ is the image of the line $\cZ(yX_0+bX_1)$ under the Veronese map, and meets the conic $\cC(R)$ in the point with coordinates $(b^2,by,0,y^2,0,0)$. Clearly the points $Q_{1,1}$ and $Q_{b,b}$ belong to the same conic plane $\langle \cC(Q_{1,1})\rangle = \langle \cC(Q_{b,b})\rangle$.

Inside $K_P$, fixing the point $R$ fixes the point $\cC(R)\cap \cC(P)$ which has coordinates $(0,0,0,1,0,0)$, the nucleus of $\cC(R)$ with coordinates $(0,1,0,0,0,0)$, the intersection of $\cC(R)$ with the tangent line of $\cC(R)$ through $R$, which has coordinates $(1,0,0,0,0,0)$, and the intersection, with coordinates $(1,1,0,1,0,0)$, of $\cC(R)$ with the line through $R$ and the point $\cC(R)\cap \cC(P)$. This implies that the stabiliser $K_{P,R}$ of $R$ 
inside $K_P$ fixes $\cC(R)$ pointwise.

To show that the planes $\sigma_{1,b,1}$, $b\neq 0$ form one $K$-orbit, it suffices to show that $K_{P,R}$ acts transitively on the set of points $\{Q_{b,b}~:~b\in \bF_q\setminus\{0\}\}$ in the conic plane $\langle \cC(Q_{1,1})\rangle$. To see this, observe that the element $\varphi_b \in K_{P,R}$ corresponding to the homology in $\PGL(3,q)$ with center $(0,0,1)$ and axis $\cZ(X_2)$, which maps $(x_0,x_1,x_2)$ to $(x_0,x_1,bx_2)$, maps $Q_{1,1}$ with coordinates $(1,0,1,1,1,0)$ to $Q_{b,b}$ with coordinates $(1,0,b,1,b,0)$. Therefore $\varphi_b(\sigma_{1,1,1})=\sigma_{1,b,1}$. This orbit, denoted $\Sigma_{22}$, is represented by
$$
\Sigma_{22}~:~\begin{bmatrix} x & x+z& z \\ x+z  &z  & y\\ z & y & \cdot \end{bmatrix}.
$$

\bigskip

We have thus proved that the set $\mathcal S$ contains exactly two $K$-orbits of planes meeting the nucleus plane in the point $P$, namely $\Sigma_{21}$ and $\Sigma_{22}$.

\bigskip

\underline{The set $\mathcal T$.}

Planes in $\mathcal T$ intersect the secant variety of $\cV(\bF_q)$ in the cubic ${\mathcal{K}}(\tau_{a,b,1})$, which is the zero locus of the form
$$
(X+bZ)(aYZ+(X+bZ)Z)+Y(XY+aZ(X+bZ))=X^2Z+b^2Z^3+XY^2.
$$
If $b=0$ then the cubic ${\mathcal{K}}(\tau_{a,b,1})$ is the union of a conic $\cZ(XZ+Y^2)$ and its tangent line $\cZ(X)$ at the point parametrised by $(x,y,z)=(0,0,1)$, and $OD_0(\tau_{a,0,1})=[0,1,2q,q^2-q]$. The line $\cZ(Y)$ is a line of type $o_{12,3}$ and each two planes $\tau_{a,0,1}$ with $\operatorname{Tr}(a)=1$ belong to the same $K$-orbit by Lemma \ref{lem:$K(o_{12,3})$-action}.

If $b\neq 0$ then the cubic ${\mathcal{K}}(\tau_{a,b,1})$ is irreducible and consists of points parameterised by $(x,y,z)$ belonging to the set
$$\{(0,1,0)\} \cup \{(1,(z+b^2z^3)^{1/2},z)~:~z\in \bF_q\},$$
and $OD_0(\tau_{a,b,1})=[0,1,q,q^2]$. The plane $\tau_{a,b,1}$ meets the conic plane $\langle \cC(\pi)\rangle$ in the point 
$(b,a,0,1,0,0)$, and so necessarily $b\neq a^2$.
Let $R'$ denote the unique double point of ${\mathcal{K}}(\tau_{a,b,1})$ parameterised by $(x,y,z)=(b,0,1)$. Then the line $\ell'=\langle P,R'\rangle$ is of type $o_{13,1}$, and $H(R')=H(\pi)=\cZ(Y_5)$. By the above this means that the plane $\tau_{a,b,1}$, $b\neq 0$, is  equivalent to the plane $\sigma_{1,1,1}\in {\mathcal{S}}$.
This defines a unique orbit $\Sigma_{23}$ represented by $$
\Sigma_{23}~:~\begin{bmatrix} x & az& x \\ az  &z  & y\\ x & y & \cdot \end{bmatrix}, 
$$ 
$\operatorname{Tr}(a)=1$, which completes the proof.
\end{proof}

\begin{Remark}\label{mainremark}
Planes in $\Sigma_{16}\dots\Sigma_{23}$ can be distinguished by their point-orbit distributions, except for those in $\Sigma_{20}\cup\Sigma_{22}$ and $\Sigma_{21}\cup\Sigma_{23}$. In these cases, the geometry of the associated cubic curves provides a complete invariant. Specifically, $\mathcal{C}_{20}$ is the union of a line and an imaginary pair of lines, $\mathcal{C}_{22}$ is irreducible, $\mathcal{C}_{21}$ is the union of a line and a double line, and $\mathcal{C}_{23}$ is the union of a non-singular conic with its tangent line.
  










\end{Remark}

\begin{Theorem}\label{thm:point_in_nucleus_plane}
There are six $K$-orbits on planes $\pi$ in $\PG(5,q)$ which are disjoint from the Veronese surface $\cV(\bF_q)$, and meet the nucleus plane of $\cV(\bF_q)$ in a point: (a) one $K$-orbit of planes with point-orbit distributions $OD_0(\pi)=[0,1,0,q^2+q] $, (b) one $K$-orbit with $OD_0(\pi)=[0,1,3q,q^2-2q]$, (c) two $K$-orbits with $OD_0(\pi)=[0,1,q,q^2] $, and (d) two $K$-orbits with $OD_0(\pi)=[0,1,2q,q^2-q]$.
\end{Theorem}
\begin{proof} 
The result is immediate from Theorem \ref{thm:same_hyperplanes}, Theorem \ref{thm:different_hyperplanes} and Remark \ref{mainremark}.
\end{proof}

The next theorem gives the classification of planes which are disjoint from the Veronese surface but are not disjoint from the nucleus plane.

\begin{Theorem}
    Let $\pi$ be a plane in $\PG(5,q)$ such that $\pi\cap\cV(\Fq)=\emptyset$ and $\pi\cap \pi_{\mathcal{N}}\neq \emptyset$. Then, one of the following cases holds:
    \begin{enumerate}[(i)]
        \item $\pi=\pi_{\mathcal{N}}$,
        \item $\pi\cap\pi_{\mathcal{N}}=\ell\in o_{12,1}$. In this case, $\pi$ has either $0$ or $q$ rank-$2$ points outside the nucleus plane. This defines two $K$-orbits of planes with point-orbit distributions $[0,q+1,0,q^2] $ and $[0,q+1,q,q^2-q]$,
        \item $\pi\cap\pi_{\mathcal{N}}=P$. In this case, $\pi$ has either $0$, $q$, $2q$ or $3q$ rank-$2$ points outside the nucleus plane. This defines (a) one $K$-orbit of planes with point-orbit distributions $OD_0(\pi)=[0,1,0,q^2+q] $, (b) one $K$-orbit with $OD_0(\pi)=[0,1,3q,q^2-2q]$, (c) two $K$-orbits with $OD_0(\pi)=[0,1,q,q^2] $, and (d) two $K$-orbits with $OD_0(\pi)=[0,1,2q,q^2-q] $.
    \end{enumerate}
\end{Theorem}
\begin{proof} 
    Case $(i)$ is obvious. 
    Case $(ii)$ is Theorem \ref{thm:line_in_nucleus_plane} and case $(iii)$ is Theorem \ref{thm:point_in_nucleus_plane}.
\end{proof}

\begin{Remark}
There are $18$ $\alpha(\PGL(3, 2))$-orbits of planes in $\PG(5,2)$ intersecting $\pi_{\mathcal{N}}$ non-trivially. The full automorphism group of $\cV(\mathbb{F}_2)$ is $\operatorname{Sym}_7$, which  does not preserve the rank distribution (see \cite[Remark 3.30]{CSRD codes}). Under $\operatorname{Sym}_7$, there are $9$ orbits of planes in $\PG(5,2)$. These numbers can be verified computationally using the FinInG package \cite{fining} in GAP \cite{GAP}.
\end{Remark}

\section{Comparison with previous work}\label{comparison}

In \cite{campbell2}, Campbell identified several non-equivalent nets of conics in $\PG(2,q)$ for $q$ even, often without proof. Moreover, the work does not provide a complete classification of the equivalence classes. Some shortcomings in Campbell’s treatment have been pointed out in \cite{zanella,planesqeven}. For instance, Zanella, in \cite{zanella}, constructed nets containing $q^2 + q + 1$ non-singular conics for all $q$, thereby disproving a claim in \cite{campbell2} that such nets exist only when $q \equiv 1 \pmod{3}$. Additionally, some arguments in \cite{campbell2} (see, for example, \cite[pp.~482]{campbell2}) rely on Campbell's earlier assertion in \cite{campbell} that pencils with $q$ non-singular conics and a unique pair of conjugate imaginary lines do not exist. However, as we have demonstrated in \cite[Section~7]{solidsqeven}, this assertion is incorrect. In \cite{campbell2} Campbell listed $17$ classes of nets with at least one double line. Our work completes the classification of nets containing at least one double line in $\PG(2,q)$, by (i) identifying the previously unlisted $\PGL(3,q)$-orbit of nets of conics, represented by
\[
\alpha (cX_0X_2 + X_1^2) + \beta (X_0^2 + X_0X_2 + X_1X_2) + \gamma X_2^2,
\]
which corresponds to the $K$-orbit of  planes $\Sigma_{18}$, where $c$ is a non-admissible element of $\Fq$ satisfying $\operatorname{Tr}(c^{-1}) = \operatorname{Tr}(1)$, (ii) establishing the existence of these orbits, and (iii) providing a complete set of geometric-combinatorial invariants for each of these orbits.

\section*{Acknowledgements}

The first author acknowledges the support of the \textit{Croatian Science Foundation}, project number HRZZ-UIP-2020-02-5713. The second author is supported in part by the Slovenian Research and Innovation Agency (ARIS) under research program P1-0285 and research project J1-50000.


\begin{thebibliography}{10}

\bibitem{AbEmIa} N. Abdallah, J. Emsalem and A. Iarrobino, ``Nets of Conics and associated Artinian algebras of length 7", {\it European Journal of Mathematics}. \textbf{9}, (2023), 238583607.



\bibitem{CSRD codes} N. Alnajjarine, M. Lavrauw. “Linear complete symmetric rank-distance codes", submitted, arXiv: 2503.02586.

\bibitem{NourMichel}
N.~Alnajjarine and M.~Lavrauw, “Determining the rank of tensors in $\bF_q^2 \otimes \bF_q^3 \otimes \bF_q^3$”, in: D. Slamanig, E. Tsigaridas and Z. Zafeirakopoulos (eds.),  MACIS 2019: Mathematical Aspects of Computer and Information Sciences, {\it Lecture Notes in Computer Science 11989}, Springer, Cham, 2020.



\bibitem{webs} N. Alnajjarine, M. Lavrauw. “Webs and squabs of Conics over Finite Fields",
{\it Finite Fields and Their Applications}. \textbf{102}, (2024), 102544.


\bibitem{solidsqeven}
N. Alnajjarine, M. Lavrauw, T. Popiel, “ Solids in the space of the Veronese surface in even characteristic”, {\it Finite Fields and Their Applications}. \textbf{83}, (2022), 102068.

\bibitem{planesqeven}
N. Alnajjarine, M. Lavrauw, “ A classification of planes intersecting the Veronese surface
over finite fields of even order.” {\it Designs, Codes and Cryptography}. \textbf{83}, (2023), https://doi.org/10.1007/s10623-023-01194-9.


\bibitem{fining} J.~Bamberg, A.~Betten, Ph.~Cara, J.~De~Buele, M.~Lavrauw and M.~Neunh\"offer, FinInG: Finite Incidence Geometry: FinInG -- A {GAP} package, (2018); \url{http://www.fining.org}







\bibitem{berlekamp} E. Berlekamp, H. Rumsey \& G. Solomon,  \newblock “Solutions of algebraic equations over fields of characteristic 2".
\newblock {\it Jet Propulsion Lab. Space Programs Summary}, \textbf{ 4\/}, (1966), 37--39.


\bibitem{roots} E.~Berlekamp, H.~Rumsey and G.~Solomon G, ``On the solution of algebraic equations over finite fields'', {\it Information and Control} \textbf{10}, (1967), 553--564.




\bibitem{campbell}
A.~Campbell, ``Pencils of conics in the Galois fields of order $2^n$'', {\it Amer. J. Math.} \textbf{49}, (1927), 401--406.

\bibitem{campbell2}
A. Campbell, “Nets of conics in the Galois field of order $2^n$”, {\it Bull. Amer. Math. Soc.} \textbf{34}, (1928), 481--489.



 
\bibitem{dickson}
L.~E.~Dickson, ``On families of quadratic forms in a general field'', {\it Quarterly J. Pure Appl. Math.} \textbf{45}, (1908), 316--333.

\bibitem{GAP}
The GAP Group, GAP -- Groups, Algorithms, and Programming, Version 4.11.1, (2021); \newline 
\url{https://www.gap-system.org}.

  \bibitem{Havlicek}
H.~Havlicek, ``Veronese varieties over fields with non-zero characteristic: a survey'', {\it Discrete Math.} \textbf{267}, (2003), 159--173.


\bibitem{galois geometry} J. Hirschfeld and J. Thas, \newblock {\it  General Galois geometries.} \newblock London : Springer, (1991).


\bibitem{hirsch}
J.~W.~P.~Hirschfeld, {\it Projective geometries over finite fields}, second edition, Oxford University Press, Oxford, (1998).



 \bibitem{canonical}
M.~Lavrauw and J.~Sheekey, ``Canonical forms of $2 \times 3 \times 3$ tensors over the real field, algebraically closed fields, and finite fields'', {\it Linear Algebra Appl.} \textbf{476}, (2015), 133--147.




\bibitem{lines} 
M.~Lavrauw and T.~Popiel, ``The symmetric representation of lines in $\mathrm{PG}(\Fq^3\otimes\Fq^3)$'', {\it Discrete Math.} \textbf{343}, (2020), 111775.



\bibitem{nets}
M.~Lavrauw, T.~Popiel and J.~Sheekey, ``Nets of conics of rank one in $\mathrm{PG}(2,q)$, $q$ odd'', {\it J. Geom.} \textbf{111}, (2020), 36. 
 
\bibitem{LaPoSh2021}
M.~Lavrauw, T.~Popiel and J.~Sheekey, ``Combinatorial invariants for nets of conics in $\PG(2,q)$'', {\it Des. Codes. Cryptogr.} \textbf{90}, (2022), 2021–2067.


\bibitem{T233}
M.~Lavrauw and N.~Alnajjarine, 
\it{T233: Algorithms for tensors in $\Fq^2 \otimes \Fq^3 \otimes \Fq^3$},  https://github.com/mlavrauw/T233, 2019.




\bibitem{wilson}
A. Wilson, “The canonical Types of Nets of Modular Conics”, {\it Amer. J. Math.} \textbf{36}, (1914), 187--210.

\bibitem{zanella}
C. Zanella, “On finite Steiner surfaces”, {\it Discrete Math}. \textbf{312}, (2012), 652--656.

\bibitem{wall}
C. Wall, ``Nets of conics'', {\it Math. Proc. Cambridge Philos. Soc.} \textbf{81}, (1977), 351--364.


\end{thebibliography}
\end{document}